%% file: groups.tex
\documentclass[a4paper]{article}
\usepackage[utf8]{inputenc}
\usepackage{amsmath}
\usepackage{amsthm}
\usepackage{amssymb}
\usepackage{pgf}
\usepackage{tikz}
\usetikzlibrary{automata}
\usetikzlibrary{decorations,arrows}
\usetikzlibrary{decorations.pathmorphing}
\usetikzlibrary{positioning}
\usepackage[british]{babel}

\newcommand\definesymb[1]{%
\expandafter\newcommand\csname #1#1\endcsname{{\ensuremath{\mathbb{#1}}}}%
}

\definesymb{Z}
\definesymb{N}
\definesymb{R}

\newtheorem{theorem}{Theorem}
\newtheorem*{theorem*}{Theorem}
\newtheorem{defn}{Definition}[section]
\newtheorem{proposition}[defn]{Proposition}
\newtheorem{fact}[defn]{Fact}
\newtheorem{cor}[defn]{Corollary}

\title{Aperiodic Subshifts of Finite Type on Groups}

\author{Emmanuel Jeandel\\
LORIA, UMR 7503 - Campus Scientifique, BP 239\\
54\,506 VANDOEUVRE-L\`ES-NANCY, FRANCE\\
\texttt{emmanuel.jeandel@loria.fr}}

\begin{document}

\maketitle

\begin{abstract}

In this note we prove the following results:
\begin{itemize}
	\item If a finitely presented group $G$ admits a strongly aperiodic
	  SFT, then $G$ has decidable word problem.
	  More generally, for f.g. groups that are not recursively
	  presented, there exists a computable obstruction for them to
	  admit strongly aperiodic SFTs.
	\item On the positive side, we build strongly aperiodic SFTs on
	  some new classes of groups.
	  We show in particular that some particular monster groups admits
	  strongly aperiodic SFTs for trivial reasons.
	  Then, for a large class of group $G$, we show how to build
	  strongly aperiodic SFTs over $\mathbb{Z} \times G$.
	  In particular, this is true for the free group with $2$
	  generators, Thompson's groups $T$ and $V$, $PSL_2(\mathbb{Z})$
	  and any f.g. group of rational matrices which is bounded.
\end{itemize}
\end{abstract}

While Symbolic Dynamics \cite{LindMarcus} usually studies subshifts on
$\mathbb{Z}$, there has been a lot of work generalizing these results
to other groups, from dynamicians and computer scientists working in higher dimensions
($\mathbb{Z}^d$ \cite{LindMulti}) to group theorists interested in
characterizing group properties in terms of topological or dynamical
properties \cite{CicCoo}.

In this note, we are interested in the existence of aperiodic
Subshifts of Finite Type (SFT for short), or more generally of aperiodic
effectively closed shifts.

A subshift on a group $G$ corresponds informally to a way of coloring the
elements of the group, subject to some local constraints.
The constraints are in finite number for a SFT, the constraints are
given by an algorithm for an effectively closed shift.

Of great interest is the existence of aperiodic subshifts, which are
nonempty subshifts for which no coloring has a translation vector, i.e. is
invariant under translation by a nonzero element of $G$.

An important example of a group with an aperiodic SFT is
the group $\mathbb{Z}^2$, for which SFTs are sometimes called tilings
of the plane and given by Wang tiles \cite{wangpatternrecoII}, 
the most famous example being the Robinson tiling \cite{Robinson}.

Not all group admits an aperiodic SFTs though, it is for example easy to see
that there are no aperiodic SFT over $\mathbb{Z}$. However, all
(countable) groups admits aperiodic shifts, this result is surprisingly 
nontrivial and quite recent \cite{Gao}.

There has been a lot of work proving how to build aperiodic SFTs in a
large class of groups, and more generally tilings on manifolds \cite{Mozes:1997,Cohen2014}.
It is an open question to characterize groups that admit strongly
aperiodic SFTs.
Cohen\cite{Cohen2014} proved that this property is a quasiisometry
invariant, and Carroll and Penland \cite{Carroll} proved it is a
commensurability invariant.

\paragraph{}
In the first part of this article, we use computability theory to
prove that groups admitting aperiodic SFTs must satisfy some
computability obstruction. In particular, a finitely presented group
with an aperiodic SFT has a decidable word problem. This is proven in
section \ref{sec:eff}.
Cohen\cite{Cohen2014} showed that f.g. groups admitting strongly aperiodic SFTs are
one ended and asked whether it is a sufficient condition. Our first
result proves in particular that this condition is not sufficient.

For f.g. groups that are not finitely presented, the computability
condition is harder to understand: Intuitively it means that the
information about which products of generators are equal to the
identity is enough to know which products aren't, i.e. we can obtain
negative information about the word problem from positive information.
The exact criterion is formulated precisely using the concept of
enumeration reducibility. This  generalization is presented in section
\ref{sec:enum}, and might be omitted by any reader not familiar or
interested with recursion theory.

The more general result we obtain is as follows:
\begin{theorem*}
If a f.g. group $G$ admits a normally aperiodic effectively closed
subshift, then the complement of the word problem of $G$ is
enumeration reducible to the word problem of $G$.

For finitely presented groups, this implies the word problem of $G$ is
decidable.
\end{theorem*}
Normal aperiodicity is a weakening of the notion of
aperiodicity, which is intermediate between the notion of weak
aperiodicity and of strong aperiodicity.
Strongly aperiodic subshifts ask that the stabilizer of each point is
finite. Here we ask that the stabilizer of each point does not contain
a normal subgroup.

The first part of the article is organized as follows. Subshifts can
be defined as shift-invariant topologically closed sets on $A^G$, the
sets of functions from $G$ to $A$.
The first section introduces an effective notion of closed sets on
$A^{\mathbb{F}_p}$ and then on $A^G$, and proves some link between the
two.
In the second section, we use this effective notion to prove the main
result for recursively presented groups
In the third section, we use concepts of computability theory to
generalize the results to any f.g. group.

\paragraph{}
In the second part of the article, we exhibit aperiodic subshifts of
finite type for some groups $G$.
We first show that some monster groups, which are infinite simple
groups with some bad properties, admits strongly aperiodic SFTs
The SFT we obtain are quite trivial and degenerate.

In the last section, we will remark how a variation on a technique by
Kari gives aperiodic SFTs on $\mathbb{Z} \times G$ for a large class
of group $G$. We do not know if there exists an easier proof of this
statement.
This class of groups contains the free group, f.g. subgroups of
$SO_n(\mathbb{Q})$, and Thompson's group $T$ and $V$.


\clearpage
\section{Effectively closed sets on Groups}

We first give definitions of effectively closed sets, which are some particular
closed subsets of the Cantor Space $A^G$ and $A^{\mathbb{F}_p}$.
The reader fluent with symbolic dynamics should remark that the sets we consider are not supposed to be
translation(shift)-invariant in this section.

\subsection{Effectively closed sets on the free group}

Let $\mathbb{F}_p$ denote the free group on $p$ generators, with
generators $x_1 \dots x_p$.
Unless specified otherwise, the identity on $\mathbb{F}_p$, and any
other group will be denoted by $\lambda$, and the symbol $1$ will be used only for
denoting a number.
Let $A$ be a finite alphabet.

A \emph{word} is a map $w$ from a finite part of $\mathbb{F}_p$ to $A$.
A configuration $x \in A^{\mathbb{F}_p}$ \emph{disagrees} with a word $w$
if there exists $g$ so that $x_g \not= w_g$ and both sides are well defined.

\begin{defn}
Let $L$ be a list of words.
The \emph{closed set} defined by $L$ is the subset $S^{\mathbb{F}_p}_L$ of $A^{\mathbb{F}_p}$ of all
configurations $x$ that disagree with all words in $L$.

Such a set is a closed set for the prodiscrete topology on $A^{\mathbb{F}_p}$.

if $S$ is a closed set of $A^{\mathbb{F}_p}$, we denote by ${\cal L}(S)$ the
set of all words of $\mathbb{F}_p$ that disagree with $S$.
Note that $S = S_{{\cal L}(S)}$.
\end{defn} 

\begin{defn}
A closed set $S$ of $A^{\mathbb{F}_p}$ is effectively closed if $S =
S_L$ for a recursively enumerable set of words $L$.
Equivalently, ${\cal L}(S)$ is recursively enumerable.
\end{defn}	
The equivalence needs a proof. We give it in the form of the following
result, which will be useful later on.
\begin{proposition}
	\label{prop:empty}
	There exists an algorithm that, given an effective enumeration
	$L$, halts iff $S_L^{\mathbb{F}_p}$ is empty.
\end{proposition}
\begin{proof}
	For a finite set $L$, it is easy to test if $S_L^{\mathbb{F}_p}$ is empty:
	just test all possible words of $A^{\mathbb{F}_p}$ defined on the union of
	the supports of all words in $L$. 
	
    Furthermore, by compactness, for an infinite $L'$, $S_{L'}^{\mathbb{F}_p}$ is empty iff there
	exists a finite $L \subseteq L'$ so that $S_{L}^{\mathbb{F}_p}$ is empty.
	
    Now, if $L$ is effective, 
	consider the following algorithm: enumerate all elements $w_i$ in
	$L$, and test at each step if $S_{\{ w_1, \dots, w_n\}}^{\mathbb{F}_p}$ is empty.
	By the first remark, it is indeed an algorithm. By the second
	remark, this algorithm halts iff $S_{L}^{\mathbb{F}_p}$ is empty.
\end{proof}	
\begin{proof}[Proof of the equivalence in the definition]
Suppose that $S = S_L^{\mathbb{F}_p}$ is effectively closed.
We will prove that ${\cal L}(S)$ is recursively enumerable.
Let $w$ be a finite word. Suppose that $w$ is defined over $I$.
Let $W$ be the set of all words defined over $I$ that are incompatible
with $w$. Then $w \in {\cal L}(S)$ iff $S_{L \cup W} = \emptyset$.

\end{proof}	
\clearpage
\subsection{Effectively closed sets on groups}

We will now look at closed sets of $A^G$.
From now on, a given finitely generated group $G$ will always be given
as a quotient of a free group, that is $G = \mathbb{F}_p / R$ for $R$
a normal subgroup of $\mathbb{F}_p$.
As long as $G$ is finitely generated, it is routine to show that all
such representations give the exact same definition of effectiveness.

Let $\phi$ be the natural map from $\mathbb{F}_p$ to $G$.
For $g,h \in \mathbb{F}_p$ we will write $g =_G h$ for $\phi(g) =
\phi(h)$.

Let $w$ be a word on $\mathbb{F}_p$. We say that $w$ is a $G$-word if
$w_g = w_h$ whenever $g =_G h$ and both sides are defined.

A configuration $x \in A^{G}$ \emph{disagrees} with a word $w$
if there exists $g \in \mathbb{F}_p$ so that $x_{\phi(g)} \not= w_g$ and both sides are  defined.
Note that a configuration in $A^G$ always disagrees with a word
which is not a $G$-word.

\begin{defn}
	The closed set on $G$ defined by $L$ is the subset $S^{G}_L$ of $A^G$
of all configurations $x$ that disagree with all words in $L$.

Such a set is a closed set for the prodiscrete topology on $A^G$.

If $S$ is a closed set of $A^G$, we denote by ${\cal L}(S)$ the
set of all words of $\mathbb{F}_p$ that disagree with $S$.
Note that $S = S_{{\cal L}(S)}$.
\end{defn}
Note in particular that all words that are not $G$-words are always
elements of ${\cal L}(S)$.

\begin{defn}
A closed set $S$ of $A^G$ is effectively closed if $S = S^G_L$ for a recursively enumerable set of words $L$.

\end{defn}	
Note that this is no longer equivalent to ${\cal L}(S)$ being recursively enumerable.
Indeed, $S= A^G$ is effectively closed but ${\cal L}(A^G)$ is 
recursively enumerable only if $G$ is recursively presented (see below).

In the following we will see closed sets of $A^G$ as closed sets of
$A^{\mathbb{F}_p}$.

To do so, denote by 
 $Per_G$ the set of all
configurations of $A^{\mathbb{F}_p}$ which are $G$-consistent, that is
$x_{g} = x_{h}$ whenever $g =_G h$. It is clear that $Per_G$ is a
closed set: In fact ${\cal L}(Per_G)$ is exactly the set of words that are
not $G$-words.

Furthermore $Per_G$ is isomorphic to $A^G$: There
is a natural map $\psi$ from $Per_G$ to $A^G$ defined
by $\psi(x) = y$ where $y_{\phi(g)} = x_g$. $\psi$ is invertible with
inverse defined by $\psi^{-1}(y) = x$ where $x_g = y_{\phi(g)}$.

This bijection preserves the closed sets in the following sense:

\begin{fact}
\[
	\psi(S_L^{\mathbb{F}_p} \cap Per_G) = S_L^G \]

\[ {\cal L}(S_L^{\mathbb{F}_p} \cap Per_G) = {\cal L}(S_L^G)\]
\end{fact}	

While $S_L^{\mathbb{F}_p}$ is always effective if $L$ can be enumerated, it might be possible for 
$S_L^{\mathbb{F}_p} \cap Per_G$ to not be effective.
 In fact:
\begin{proposition}
	\label{prop:eff:rec}
	Let $A$ be an alphabet of size at least $2$.
	
	$Per_G$ is effective iff $G$ has a recognizable word problem.
\end{proposition}
A recognizable word problem means that there is an algorithm that,
given a word $w$ in $\mathbb{F}_p$, halts iff $w =_G \lambda$.
This is equivalent to saying that $G$ is recursively presented.
\begin{proof}
	If $G$ has a recognizable word problem, we can enumerate
	all pairs ${(g,h) \in \mathbb{F}_p}$ s.t. $g =_R h$, and thus enumerate
	the set of words that are not $G$-words, that is ${\cal L}(Per_G)$.

	Conversely, suppose that ${\cal L}(Per_G)$ is enumerable.
Let $\{a,b\}$ be some fixed letters
				from $A$, with  $a\not= b$.				
				For $g \in G$,	consider the word $w$ over $\{\lambda,g\}$ defined by $w_\lambda = a, w_g =  b$.
    Then $g =_G \lambda$ iff  $w \in {\cal L}(Per_G)$.
	
\end{proof}
\begin{cor}
	\label{cor:per}
If $G$ has a recognizable word problem and $S_L^G$ is effectively closed, then
$S_L^{\mathbb{F}_p} \cap Per_G$ is effectively closed.
\end{cor}	

\section{Effective subshifts on Groups}

\label{sec:eff}
If $x \in A^G$, denote by $gx$ the configuration of $A^G$ defined by $(gx)_h = x_{g^{-1}h}$.
This defines an action of $G$ on $A^G$.

\begin{defn}
A closed set $X$ of $A^G$ is said to be a subshift if $x \in X, g \in G$ implies that $gx \in X$.

$X$ is an effectively closed subshift if $X$ is effectively closed and is a
subshift.

$X$ is a SFT if there exists a finite set $L$ so that $X= S_{\{
	g^{-1}w, g \in G, w \in L\}} = S_{\{	g^{-1}w, g \in \mathbb{F}_p, w \in L\}}$.
	In particular a SFT is always effectively closed.
\end{defn}
\begin{fact}
If $X$ is a subshift, 	$\psi^{-1}(X) \subseteq Per_G$ is a subshift.
\end{fact}
Hence any subshift of $A^G$ lifts up to a subshift of $A^{\mathbb{F}_p}$.

As a warmup to the theorems, we consider $X_{\leq 1}$, the subset of $\{0,1\}^G$ of configurations
that contains at most one symbol $1$.
It is easy to see that $X_{\leq 1}$ is closed, and a subshift.

\begin{proposition}
	\label{prop:leq1}
Suppose that $G$ has a recognizable word problem.
	
	If $X_{\leq 1}$ is effectively closed then the word problem on	$G$ is decidable.
\end{proposition}

\begin{proof}
	$X_{\leq 1}$ lifts up to a subshift $Y$ on $\mathbb{F}_p$ with the property
	that (a) $Y$ is effective (as $G$ is recursively presented), hence
	${\cal L}(Y)$ is enumerable
	(b) $Y$ consists of all configurations so that $x_g = x_h = 1\implies g = _G h$.

    Now let $g \in \mathbb{F}_p$.	
	Let $w$ be the word defined by $w_\lambda = 1$ and $w_g = 1$.
	Then $w \in {\cal L}(Y)$ iff $g \not=_G \lambda$.
	
	The complement of the word problem is recognizable, therefore decidable.
\end{proof}

In the following, we are now interested in aperiodic subshifts.
\begin{defn}
	For $x \in A^G$ denote by $Stab(x) = \{ g | gx = x\}$.
			
	A (nonempty) subshift $X$ is strongly aperiodic iff
	for every $x \in X$, $Stab(x)$ is finite.

	A (nonempty) subshift $X$ is normally aperiodic iff
	for every $x \in X$, $\cap_{h \in G} Stab(hx)$ is finite.
\end{defn}
Note that there are conflicting definitions in the literature for strong aperiodicity:
Some require that $Stab(x) = \{\lambda\}$ for all $x \in X$.
The only result in this paper where this makes a difference is Prop~\ref{prop:tarski}.

Both properties are equivalent for commutative groups. 
$\cap_{h\in G} Stab(hx)$ will be called the \emph{normal stabilizer}
of $x$. It is indeed a normal subgroup of $G$, and the union of all
normal subgroups of $Stab(x)$.

Our first result states that a strongly aperiodic effectively closed subshift
(and in particular a strongly aperiodic SFT) 
forces the group to have a decidable word problem in the
class of torsion-free recursively presented group.
The next proposition strengthens the result by deleting the
torsion-freeness requirement.
\begin{proposition}
	Let $G$ be a torsion-free recursively presented group.
	
If $G$ admits a strongly aperiodic effective subshift,
then $G$ has decidable word problem.
\end{proposition}	
\begin{proof}
Let $X$ be the strongly aperiodic effective subshift.
$X$ lifts up to a subshift $Y$ on $A^{\mathbb{F}_p}$.

Let $R = \{ g | g =_G \lambda\}$.

Note that if $\psi(y) = x$, then
$Stab(y) = Stab(x) R = R Stab(x)$.
Furthermore, if $G$ is torsion-free, then $Stab(x) = \{\lambda\}$, hence for
all $y \in Y$, $Stab(y) = R$.

Now	let $g \in \mathbb{F}_p$.
Let $Z = \{ x | \forall t, x_{gt} = x_t\} = \{ x \in A^{\mathbb{F}_p} | g \in  Stab(x)\}$.

It is easy to see that $Z$ is effective. Furthermore $Y \cap Z =
\emptyset$ iff $g \not=_G \lambda$.

As there is a semialgorithm to test whether $Y \cap Z = \emptyset$, we
get that the complement of the word problem is
recognizable, hence decidable.
\end{proof}

\begin{proposition}
	Let $G$ be a recursively presented group.

    If $G$ admits a normally aperiodic effectively closed subshift,
    then it admits a normally aperiodic effectively closed subshift $X$ where for all
    $x \in X$, ${\cap_{h \in G} Stab(hx) = \lambda}$.
\end{proposition}

\begin{proof}
	Let $X$ be normally aperiodic.
	
	The proof is in two steps.
	In the first step, we will prove that there exists a finite normal
	subgroup $H$ of $G$ and a nonempty effective subshift $Y$ so that for all $x
	\in Y$, $\cap_{g \in G} Stab(gx) = H$.
	
	Let $H_0 = \{\lambda\}$.
	Suppose that there exists $x \in X$ so that
	$H_0 \subsetneq \cap_{h \in G} Stab(hx)$.
	Then let's denote $H_1 \supset H_0 $ the normal subgroup on the right.
	
    We do the same for $H_1$, building progressively a chain of normal
	subgroups $H_1 \dots H_n \dots$.

    It is impossible however to obtain an infinite chain this way.
	Indeed, as for all $i$, there exists $x_i$ so that $\cap_{g \in G}
	Stab(gx_i) = H_i$, a limit point $x$ of $x_i$ would verify 
	$\cap_{g \in G} Stab(gx) \supseteq \cup_i H_i$, hence $x$ would be
	a configuration with an infinite normal stabilizer, impossible by definition.
	
	Hence this process will stop, and we obtain some finite normal
	subgroup $H$ of $G$ and a point $x_0$ so that 
	$\cap_{h \in G}	Stab(hx) = H$ and no point $x$ has a larger normal stabilizer.

	Now let $Y = \{ x \in X | \forall g \in G, h\in H, hgx = gx \}$.
    $Y$ is nonempty, as it contains $x_0$.
	As $H$ is finite, $Y$ is clearly effectively closed. As $H$ is
	normal, it is a subshift. Furthermore, for all $x \in Y$,
	$\cap_{h \in G}	Stab(hx) = H$.

\hspace{5mm}

    Now the second step.
	Take $Z = \{ x \in H^G | \forall g \in G, h \in H \setminus \{ \lambda\},  x_{hg}\not= x_{g}\}$.
	$Z$ is clearly effectively closed. As $H$ is normal, it is a subshift\footnote{
	Indeed, let $z \in Z$ and $t \in G$.
	Let $g \in G$ and $h \in H \setminus \{\lambda\}$.
	Then $(tz)_{hg} = z_{t^{-1}hg} = z_{t^{-1}htt^{-1}g} \not=
	z_{t^{-1}g} = (tz)_g$, hence $tz \in Z$}.
	$Z$ is also nonempty: Write $G = HI$ where
    $I$ is a family of representatives of $G/H$. Then
	the point $z$ defined by $z_g = h$ if $g \in hI$ is in $Z$, hence
	$Z$ is nonempty\footnote{
  Indeed, let $g \in G$ and $h \in H \setminus \{\lambda\}$.
  Then $z_g = k$ where $g \in kI$ for some $k \in H$.
  But $hg \in (hk)I$ hence $z_{hg} = hk \not= k = z_g$.
  }. Furthermore, if $z \in Z$, then $Stab(x) \cap H = \{\lambda\}$
\footnote{
  Indeed, for $h \in H \setminus \{\lambda\}$,
  $(hx)_\lambda = x_{h^{-1}} \not= x_\lambda$, hence $h \not\in Stab(x)$.  
  }.
As a consequence, $Z \times Y$ is a nonempty subshift for which for
all $x \in Z \times Y$, $\cap_{h \in G} Stab(hx) = \{\lambda\}$.
\end{proof}

\begin{theorem}
	\label{cor:final}
	Let $G$ be a recursively presented group.

	If $G$ admits a normally aperiodic effectively closed subshift, $G$ has decidable word problem.
\end{theorem}
\begin{proof}
This is more or less the same proof as before, with one slight
difference.

Let $X$ be the normally aperiodic effectively closed subshift. We may
suppose by the previous proposition that for all $x \in X$, $\cap_{h\in G} Stab(hx) =
\{\lambda\}$.

$X$ lifts up to a subshift $Y$ on $A^{\mathbb{F}_p}$ with the following property:
If $y \in Y$, then $\cap_{h \in \mathbb{F}_p} Stab(hy) = R$.

Now	let $g \in \mathbb{F}_p$.
Let $Z = \{ y | \forall h \in \mathbb{F}_p, ghy = hy\} = \{ y | g \in \cap_{h\in \mathbb{F}_p} Stab(hy)\}$.
$Z$ is effective.
Furthermore $Y \cap Z = \emptyset$ iff $g \not=_G \lambda$.

Emptyness is recognizable, hence the complement of the word problem is
recognizable, hence decidable.

\end{proof}

\begin{cor}
Let $G$ be a recursively presented group that admits a strongly
aperiodic SFT. Then $G$ has a decidable word problem.	
\end{cor}

\section{Enumeration degrees}
\label{sec:enum}
In this section, we generalize the previous results to any finitely
generated groups, whose presentation might be not recursive.

Recall that the word problem of $G$ is the set 
\[
W(G) = \left\{ g \in \mathbb{F}_p \middle| g =_G \lambda \right\}
\]
If $G$ is any f.g. group, there are various ways that a subshift $X = S_L$
could be said to be computable in $G$:
\begin{itemize}
	\item We can enumerate a list of forbidden pseudo-words for $S$
	\item We can enumerate a list of forbidden pseudo-words for $S$
	  give an oracle for $W(G)$
	\item We can enumerate a list of forbidden pseudo-words for $S$
	  give a list of all elements of $W(G)$
\end{itemize}	
The last two are possibly different: In the last case, we are only
given access to positive informations, i.e. the elements that are in
$W(G)$, not those that aren't.

The definitions and results below involve the notion of 
\emph{enumeration reducibility} \cite{FriedbergRogers,Odifreddi2}.

Enumeration degrees, and enumeration reducibility, is a notion from
computability theory that is quite natural in the context of presented groups
and subshifts, as it captures  (in computable terms) the fact that the
only information we have about these objects are positive (or negative) information: In a subshift (effective or not), we usually have
ways to describe patterns that do not appear, but no procedure to
list patterns that appear. In a presented group, we have information
about elements that correspond to the identity element of the group,
but no easy way to prove that an element is different from the identity.

This reduction has been used already  in the context of groups
\cite[Chapter 6]{HigmanScott} or in symbolic dynamics \cite{AubrunS09}.

\subsection{Definitions}

If $A$ and $B$ are two sets of numbers (or words in $\mathbb{F}_p$), we say that
$A$ is enumeration reducible to $B$ if there exists an algorithm that
produces an enumeration of $A$ from any enumeration of $B$.
Formally:
\begin{defn}
   $A$ is enumeration reducible to $B$, written $A \leq_e B$, if there
   exists a partial computable function $f$ that associates to each $(n,i)$ a
   finite set $D_{n,i}$ s.t.
   $n \in A \iff \exists i, D_{n,i} \subseteq B$.
\end{defn}

We will first give here a few easy facts, and then examples relevant
to group theory and symbolic dynamics.

\begin{fact}
	$A$ is recursively enumerable iff $A \leq_e \emptyset$.
	
In particular, if $A$ is recursively enumerable, then $A \leq_e B$ for all $B$.
 
 $A \leq_e B \oplus \overline{B}$ iff $A$ is enumerable given $B$ as
 an oracle.
\end{fact}	
$B \oplus \overline{B}$ denotes the set $\{(1,x) | x \in B\} \cup \{(1,x) | x \not\in B\}$.

Here are some examples relevant to group theory:
\begin{fact} (Formal version) Let $G=\left<X | R\right>$ be a finitely generated group,
	  with $R \subseteq \mathbb{F}_p$ and $N$ be the normal subgroup of $\mathbb{F}_p$
	  generated by $R$. Then $N \leq_e R$.
	  In particular, if $R$ is finite, then $N$ (hence the word
	  problem over $G$) is recursively enumerable.
	  
	  (Informal version) From a presentation $R$ of a group, we can list
	  all elements that correspond to the identity element of the group
	  (but in general we cannot list elements that are not identity of the group).
	  In terms of reducibility, the set of all elements that correspond to the
 	  identity is the smallest possible presentation of a group.
\end{fact}

Indeed $g \in N$ iff there exists $g_1 \dots g_k \in \mathbb{F}_p,  u_1
\dots u_k \in R \cup R^{-1}$ so that $g = g_1 u_1 g_1^{-1} g_2  u_2 g_2^{-1} \dots g_k u_k g_k^{-1}$.
Given any enumeration of $R$ (and as $\mathbb{F}_p$ is enumerable), we can  therefore enumerate $N$.

Here are some examples relevant to symbolic dynamics or topology.
\begin{fact}
(Formal version)
	Let $S = S_L$ be any closed set.
	Then ${\cal L}(S) \leq_e L$.

In particular if $L$ is enumerable then ${\cal L}(S)$ is recursively enumerable.

(Informal version) From any description of a closed set in terms of some forbidden
words, we may obtain a list of all words that do not appear (but
usually not of patterns that appear). In terms of reducibility, 
the set of all words that do not appear is the smallest possible description of
a closed set.
\end{fact}
Subshifts are particular closed sets, so this is also true for
subshifts. In particular the set of patterns that do not appear in a
SFT (over $\mathbb{Z}$, or $\mathbb{F}_p$) is recursively enumerable.
\begin{proof}
This is a straightforward generalization of Prop. \ref{prop:empty} and
the subsequent proof.

Let $w$ be any word, defined over a finite set $B$.
Let $W$ be the set of all words defined over $B$ that are incompatible
with $w$. Then $w \in {\cal L}(S)$ iff $S_{L\cup W} = \emptyset$.

Now $S_{L \cup W} = \emptyset$ iff there exists a finite set $L' \subseteq L$ s.t. $S_{L' \cup W} = \emptyset$.

Thus, if $(F(n,w))_{n \in \mathbb{N}}$ is the computable enumeration of
all finite sets of words s.t that $S_{F(n,w) \cup W} = \emptyset$, then $w \in  {\cal L}(S)$ iff $\exists n, F(n,w) \subseteq L$.
\end{proof}

\subsection{Generalizations}
Now we explain how this concept gives generalizations of the previous
theorems.

First, we look at subsets of $A^{\mathbb{F}_p}$  that are effective given
an enumeration of $B$. This definition is nonstandard:
\begin{defn}
A set $S \subseteq A^{\mathbb{F}_p}$ is \emph{$B$-enumeration-effective} if
$S = S_L$ for some set of words $L$ so that $L \leq_e B$.
\end{defn}

Here are a few examples:
\begin{itemize}
	\item $\{ x \in\{0,1\}^{\mathbb{F}_p} | \forall h \in B, x_h = 1\}$ is $B$-enumeration effective
	\item $\{ x \in\{0,1\}^{\mathbb{F}_p} | \forall g \in \mathbb{F}_p, \forall h \in B, x_{gh} = x_h\}$ is $B$-enumeration effective
	\item $\{ x \in\{0,1\}^{\mathbb{F}_p} | \forall h \not\in B, x_h = 1\}$ is
	  usually not $B$-enumeration effective. It is $B$-enumeration effective iff the complement of $B$ is enumeration reducible to $B$.
	  It happens for example whenever the complement of $B$ is
	  enumerable, regardless of the status of $B$.
\end{itemize}	
If we go back to the different ways to define a subshift from a
group given at the introduction of this section, they correspond to
three different notions of enumeration-effectiveness:
\begin{itemize}
	\item The first one corresponds to $\emptyset$-enumeration-effective
	\item The second one corresponds to $W(G) \oplus \overline{W(G)}$-enumeration-effective
	\item The third one corresponds to $W(G)$-enumeration-effective
\end{itemize}	
The second notion is what is called a $G$-effective subshift in the
vocabulary of \cite{AubBar}.

\begin{defn}
We will say	that $X \in A^{\mathbb{F}_p}$ is $G$-enumeration effective whenever $X$ is
	$W(G)$-enumeration effective.
\end{defn}	
\begin{proposition}[Analog of 	Prop~\ref{prop:eff:rec}]

	Let $A$ be an alphabet of size at least $2$.	
	Then $Per_G$ is $G$-enumeration effective.	
	More precisely, ${\cal L}(Per_G) \equiv_e W(G)$.
\end{proposition}

\begin{proposition}
If $X,Y$ are two closed sets in $A^\mathbb{F}_p$, then
	${{\cal L}(X \cap Y) \leq_e {\cal L}(X) \cup {\cal L}(Y)}$
\end{proposition}
\begin{proof}
$X \cap Y = S_{{\cal L}(X) \cup {\cal L}(Y)}$.
\end{proof}	
\begin{cor}[Analog of Cor.~\ref{cor:per}]
	If $X$ is an effectively closed set of $A^G$,then  ${\cal L}(X) \leq_e W(G)$.
	More generally, if $X = S^G_L$ then ${\cal L}(X) \leq_e L \oplus W(G)$.
\end{cor}
\begin{proof}	
${\cal L}(X) = {\cal L}(S^G_L) = {\cal L}(S^{\mathbb{F}_p}_L \cap Per_G) \leq_e L \cup {\cal L}(Per_G) \leq_e L \oplus W(G)$.
\end{proof}	

\begin{proposition}[Analog of Prop.~\ref{prop:leq1}]
If  $X_{\leq 1}$ is effectively closed (in particular if it is an SFT) then $\overline{W(G)}  \leq_e W(G)$
\end{proposition}
\begin{proof}
$X_{\leq 1}$ lifts up to a subshift $Y$  of $A^{\mathbb{F}_p}$ which is
$G$-enumeration effective.

By the proof of Prop.~\ref{prop:leq1}, there exists a uniform family $w_g$ of words so that
$g \not= \lambda$ iff $w_g \in {\cal L}(X)$.
This implies easily that $W(G) \leq_e {\cal L}(Y)$.
\end{proof}

\begin{theorem}[Analog of Th.~\ref{cor:final}]
	If $G$ admits a normally aperiodic effective subshift $X$, then
	$\overline{W(G)} \leq_e  W(G)$.
\end{theorem}
\begin{proof}
  From the proof of Th.~\ref{cor:final}, there exists a
  $G$-enumeration effective subshift $Y$ on $\mathbb{F}_p$, and a
  (uniform) family of effectively closed sets $X_g$  so that $Y \cap X_g = \emptyset$ iff $g  \not= \lambda$.
   
Let $(F(n,g))_{n \in \mathbb{N}}$ be a computable enumeration of all finite sets
$F$ for which there exists $G \subseteq L_g$ so that $S_{F \cup G} = \emptyset$ (this can indeed be enumerated as $L_g$ can be enumerated).

Then $g \neq_G \lambda $ iff $\exists n F(n,g) \subseteq {\cal L}(Y)$, hence $\overline{W(G)} \leq_e {\cal L}(Y)$
  \end{proof}	

Hence the existence of a strongly aperiodic SFT implies that the
complement of the word problem can be enumerated from the word
problem. This is not a vacuous hypothesis: If $G$ is recursively
presented (hence $W(G)$ is recursively enumerable), this implies that
$\overline{W(G)}$ is also recursively enumerable, hence recursive.

Another way of stating the result is that, for a group having this
property, the notion of $G$-enumeration effectiveness and
$G$-effectiveness coincide.

\section{Aperiodic SFTs on monster groups}

In this section we exhibit examples of aperiodic SFTs on some monster
groups.
All our examples are trivial and somewhat degenerate. Nevertheless
these are aperiodic SFT.

\paragraph{Simple groups}
Our last section introduces a computability obstruction that must
satisfy all groups $G$ that admits an aperiodic SFT: The complement of
the word problem must be enumeration reducible to the word problem.

There are a few well known algebraic classes of groups that satisfy
this property: In particular, all f.g. simple groups have this property.

\begin{proposition}
If $G$ is a f.g. simple group, then $G$ admits  a normally aperiodic SFT
\end{proposition}	
\begin{proof}
Fix some given nontrivial element $a$ of $G$. And define
	\[ X = \left\{ x \in \{0,1,2\}^G \middle| \forall g\in G, (gx)_{a}   \not= (gx)_\lambda \right\}\]
Thus $X$ is an SFT. It is easy to see that $X$ is nonempty.

Now let $x \in X$.
Then $(a^{-1} x)_\lambda = x_{a} \not= x_\lambda$. Thus $a^{-1} x
\not= x$, which means the stabilizer of $x$ is not the whole group. A
fortiori, the normal stabilizer of $x$ is also not the whole group: $\cap_{h \in G} Stab(hx) \not= G$.
As the left term is a normal subgroup of $G$ and $G$ is simple, this implies $\cap_{h
  \in G} Stab(hx) = \{ \lambda\}$.
\end{proof}
The proof works verbatim for any group $G$ with no infinite normal
subgroups.

\paragraph{Monster groups}

The previous SFT is only normally aperiodic. For monster groups
(which are actually examples of simple groups), we can go further, and
obtain strongly aperiodic SFT.

We will use three classes of monster groups:
\begin{itemize}
	\item Tarski monster groups \cite{Olshanskii}, for which  all nontrivial proper subgroups are
of finite order $p$
    \item The Osin monster groups \cite{Osin}, for which any two nonzero element
	  are conjugate.
	\item The Ivanov groups \cite[Th. 41.2]{Olshanskiibook}, for which every element is cyclic and
	  which contains a finite number of conjugacy classes.	  	  	  
\end{itemize}	
Note that all these groups  are f.g. and simple.

\begin{proposition}
	\label{prop:tarski}x
	If $G$ is a Tarski monster group, then $G$ admits a strongly
	aperiodic SFT.
\end{proposition}
\begin{proof}
	We take again
	\[ X = \{ x \in \{0,1,2\}^G | \forall g\in G, (gx)_{a} \not= (gx)_\lambda \}\]
	for some $a \not= \lambda$.

    The stabilizer of any element of $X$ is not $G$, thus should be
	finite.
\end{proof}

In the previous example, the stabilizer of any point is finite, but
may be nontrivial.
In the two following examples, the stabilizer is trivial.

\begin{proposition}
	If $G$ is a Osin monster group, then $G$ admits a strongly
	aperiodic SFT s.t the stabilizer of any point is trivial.
\end{proposition}
\begin{proof}
	We take again
	\[ X = \{ x \in \{0,1,2\}^G | \forall g\in G, (gx)_{a} \not= (gx)_\lambda \}\]
	for some $a \not= \lambda$.
	
Now let $x \in X$.
To prove that $Stab(x) = \{\lambda\}$, let $g \in G$, $g \not= \lambda$ s.t.  $gx = x$.
By definition of the Osin monster group, $g$ is conjugate to $a$: $hgh^{-1} = a$ for some $h$, or equivalenty $hg = ah$.

Now $(hx)_a = (hgx)_a = (ahx)_a = (hx)_{a^{-1}a} = (hx)_\lambda$. Thus
$(hx)_a = (hx)_\lambda$, contradicting the fact that $x \in X$.

Thus $Stab(x) = \{\lambda\}$.
\end{proof}	

\begin{proposition}
	If $G$ is a Ivanov monster group, then $G$ admits a strongly
	aperiodic SFT s.t the stabilizer of any point is trivial.	
\end{proposition}	
\begin{proof}
Let $a_1 \dots a_{n}$ be distinct representatives of all nontrivial conjugacy
classes.

	We take
	\[ X = \{ x \in \{0,1,2, \dots 2n\}^G | \forall g\in G \forall i, (gx)_{a_i} \not= (gx)_\lambda \}\]

   $X$ is clearly an SFT, and it is clearly nonempty.
   
Now let $x \in X$. Let $g \in G$, $g \not= \lambda$ s.t.  $gx = x$.
$g$ is conjugate to $a_k$ for some $k$. We then repeat the arguments
of the previous proof.  
\end{proof}

\section{On a Construction of Kari}

\subsection{Definitions}
Kari provided a way in \cite{Kari5} to convert a piecewise affine map into
a tileset simulating it.
We give here the relevant definitions.
First, we introduce a formalism for Wang tiles that will be easier to
deal with.
\begin{defn}
Let $G$ be a f.g. group with a set $S$ of generators.

A set of Wang tiles over $G$ is a tuple 
$(C,(\phi_h)_{h\in S}, (\psi_h)_{h\in S})$ where, for each $h$, $\phi_h, \psi_h$ are maps from $C$ to some finite set.

The subshift generated by $C$ is

\[ 
	\begin{array}{rcl}
	X_C &=& \{ x \in C^G | \forall g\in G, \forall h \in S, \phi_h(x_g) =
	  \psi (x_{gh^{-1}})\} \\
	&=&
	\{ x \in C^G | \forall g \in G, \forall h \in S, \phi_h((gx)_e) =
	  \psi ((gx)_{h^{-1}})\}\end{array}\]
\end{defn}
(The last definition proves it is indeed a subshift, and in fact a
subshift of finite type).
If $G$ has one generator (in particular if $G= \mathbb{Z} =
  \left<1\right>$), we will write $\phi$ and $\psi$ instead of
	$\phi_1$ and $\psi_1$.

\begin{defn}
	Let $cont: \{0,1\}^\mathbb{Z} \rightarrow [0,1]$ defined by 
	$cont(x) = \lim\sup_n \frac{\sum_{i \in [-n,n]} x_i}{2n+1}$
	and $disc: [0,1] \rightarrow \{0,1\}^\mathbb{Z}$ defined by 
	$disc(y)_n = \left\lfloor (n+1) x\right\rfloor - \left\lfloor n
	x\right\rfloor$.
	
	Remark that $cont(disc)(y) = y$.
\end{defn}	

\begin{theorem}[\cite{Kari5}]
	\label{thm:kari}
	Let $a,b$ be rational numbers and $f(x) = ax+b$.

	Then there exists a set of Wang tiles $(C,\phi,\psi)$ over $\mathbb{Z}$
	and two maps $out, in$ from $C$ to
	$\{0,1\}$ so that the two following properties hold
	
	\begin{itemize}
		\item For any configuration $x$ of $X_C$, $f(cont(in(x)) = cont (out(x)$)
		  
		\item For any $y \in [0,1]$ so that $f(y) \in [0,1]$, 
		  there exists a configuration $x$ of $C_G$ so that $in(x) =  disc(y)$ and $out(x) = disc(f(y))$
    \end{itemize}		
\end{theorem}
$C$ is usually seen as a set of Wang tiles over $\mathbb{Z}^2$ rather
than $\mathbb{Z}$ but this formalism is better for our purpose.

Two examples are given in Figure \ref{kari:f1}.
\newcommand\wang[4]{
        \draw[line width=0.1pt] (0,0) -- +(2,2) -- +(0,4) -- +(0,0);
        \draw[line width=0.1pt] (0,0) -- +(2,2) -- +(4,0) -- +(0,0);
\begin{scope}[xshift=4cm, yshift=4cm]
        \draw[line width=0.1pt] (0,0) -- +(-2,-2) -- +(0,-4) -- +(0,0); 
        \draw[line width=0.1pt] (0,0) -- +(-2,-2) -- +(-4,0) -- +(0,0);
\end{scope}%
                \draw (1,2) node {#1};
                \draw (3,2) node {#2};
                \draw (2,3.5) node {#3};
                \draw (2,.5) node {#4};
}

\begin{figure}
	$C_1: $
	\input x1
	$C_2: $
	\input x2
\caption{Two set of Wang tiles corresponding respectively to the 
  maps $f(x) = (2x-1)/3$ and $f(x) = (4x+1)/3$.
  The colors on each tile $c\in C$ on east,west,north,south represent
  respectively $\phi(c), \psi(c), in(c), out(c)$.
  }
\label{kari:f1}
\end{figure}

\begin{cor}
	\label{cor:fam}
Let $f_1 \dots f_k$ be a finite family of affine maps with rational
coordinates.

Then there exists a set of Wang tiles $(C,\phi,\psi)$ over $\mathbb{Z}$
and maps $in$ et $(out_i)_{1 \leq i \leq k}$ from $C$ to $\{0,1\}$ so that the two following properties hold

\begin{itemize}
  \item For any configuration $x$ of $X_C$, 
	$f_i(cont(in(x)) = cont (out_i(x)$)
	
  \item For any $y \in [0,1]$ so that $f_i(y) \in [0,1]$ for all $i$, 
	there exists a configuration $x$ of $C_G$ so that $in(x) =  disc(y)$ and $out_i(x) = disc(f_i(y))$
\end{itemize}		
\end{cor}
\begin{proof}
let $(C^i, \phi^i, \psi^i)$ be the  set of Wang tiles over
$\mathbb{Z}$ corresponding to $f_i$, with maps $out^i$ and $in^i$.

Let $C = \{ y \in \prod C^i |  \exists x\in\{0,1\}, \forall i, in^i(y^i) =  x\}$.
Let $p^i$ denote the projection from $C$ to $C^i$ and define
\[ \phi = \prod (\phi^i \circ p^i) \]	
\[ \psi = \prod (\psi^i \circ p^i) \]	
\[ in = in^1 \circ p^1 = in^2 \circ p^2 = \dots= in^k \circ p^k\]
\[ out_{i} = out^i \circ p^i\]

It is clear that $C$ satisfies the desired properties.   
\end{proof}	

\begin{cor}
	\label{cor:aff}
Theorem	\ref{thm:kari} still holds when $f$ is a piecewise affine
rational homeomorphism from $[0,1]_{/0 \sim 1}$.
As a consequence, the previous corollary also holds for a finite
familiy of piecewise affine rational homeomorphisms.
\end{cor}
Let's define precisely what we mean by a piecewise affine rational
homeomorphism from $[0,1]_{/0 \sim 1}$ to $[0,1]_{/0 \sim 1}$.

We first define a relation $\equiv$: $x \equiv y$ if $x =y$ or
$\{x,y\} = \{0,1\}$. Then a piecewise affine homeomorphism
from $[0,1]_{/0 \sim 1}$ to $[0,1]_{/0 \sim 1}$ is given by a finite family $[p_i,p_{i+1}]_{i < n}$
of $n$ intervals with rational coordinates, with $p_0 = 0$ and $p_n = 1$,
and a family of affine maps $f_i(x) = a_i x + b$ so that
\begin{itemize}	
  \item $f_i$ and $f_{i+1}$ agrees on their common boundary:
	\[ \forall i \in \mathbb{Z}/n\mathbb{Z},
	f_i(p_{i+1}) \sim f_{(i+1)} (p_{i+1})\]
  \item $f = \cup_i f_i$ is injective:
	$f(x) \sim  f(y) \implies x \sim y$
\end{itemize}	
These properties imply that $f$ is invertible, and its inverse is
still piecewise affine
\begin{proof}
	We give a proof of this easy result to prepare for another
	proof later on.
	
	Let $\cal F$ be the class of relations on $[0,1] \times [0,1]$ for which the
	theorem is true.
	$\cal F$ contains all rational affine maps.
	
	It is clear from the formalism that if  $f$ and
	$g$ are in $\cal F$, then  $f\cup g \in \cal F$, by
	taking $(C\cup C', \phi\cup\phi', \psi\cup\psi)$ once the range of $\phi$ and $\phi'$ have been made disjoint.
	In the same way, we may prove  $f \circ g \in \cal F$ and $f;g \in
	\cal F$, where $f\circ g(x) = f(g(x))$ and $f;g(x) = g(f(x))$.

Let $i \in \{0\dots n-1\}$.
The function $f_i$ is the composition of $5$ functions in $\cal F$:
\begin{itemize}
	\item $g_1 = id \cup (x \rightarrow x - 1) \cup (x\rightarrow x+1) $	   ($g_1$ is exactly the relation $\equiv$)
	\item $g_2 = (x \rightarrow p_{i+1} - x) \circ (x \rightarrow
	  p_{i+1}	  - x)$ 	  ($g_2(x) = x$, defined on $[0,p_{i+1}]$)
	\item $g_3 = (x \rightarrow x + p_i) \circ (x \rightarrow x - 	  p_i)$ 	  ($g_2(x) = x$, defined on $[p_i, 1]$)
	\item $g_4 = x \rightarrow a_i x + b_i$
	\item $g_5 = g_1$
\end{itemize}	
Hence $f_i \in {\cal F}$, and $f \in {\cal F}$.

\end{proof}

Kari first used this construction~\cite{Kari14} to obtain a aperiodic SFT of
$\mathbb{Z}^2$: Start with a piecewise rational homeomorphism $f$
with no periodic points ($f$ should be indeed over $[0,1]_{/0\sim 1}$
and not over $[0,1]$ for this to work, as any continuous map from
$[0,1]$ to $[0,1]$ has a fixed point by the intermediate value theorem).  For example, take
\[
	f(x) = \left\{\begin{array}{ccl}
	  (2x-1)/3 & \text{if }& 1/2 \leq x \leq 1\\
	  (4x+1)/3 & \text{if }& 0 \leq x \leq 1/2\\
\end{array}\right.\]

Take the set of Wang tiles over $\mathbb{Z}$ given by
the theorem, and consider it as a set of Wang tiles over $\mathbb{Z}
\times \mathbb{Z}$ by mapping $in$ and $out$ to $\phi_{(0,1)}$ and
$\psi_{(0,1)}$.
For the specific function $f$ above, we obtain the set of 22 tiles
presented in Fig.~\ref{kari:f1} (where horizontal colors corresponding to
different sets are supposed to be distinct).
Then it is easy to see that this indeed gives a SFT over
$\mathbb{Z}^2$ with no periodic points.
The same construction can be refined \cite{Kari14} to obtain aperiodic
tilesets with fewer tiles, but that is not our purpose here.

The important point is that this construction may be easily generalized: There is no need
to tile $\mathbb{Z} \times \mathbb{Z}$, we may use the exact same idea
to obtain a SFT over $\mathbb{Z} \times G$ for some groups $G$.
\clearpage
\begin{defn}
	A f.g. group $G$ is PA-recognizable iff there exists a finite set
	$\cal F$ of piecewise affine rational homeomorphisms of
	$[0,1]_{/0\sim 1}$ so that
	\begin{itemize}
	  \item(A) The group generated by the homeomorphisms is isomorphic to $G$
	  \item(B) For any $t \in [0,1]_{/0\sim 1}$, if $gf(t) = f(t)$ for all $f$, then $g = e$
	\end{itemize}		
\end{defn}
Note that by the property $(A)$ every PA-recognizable group has
decidable word problem.

\begin{theorem}
	If $G$ is PA-recognizable and infinite, there exists a SFT over $\mathbb{Z}
	\times G$ which is strongly aperiodic.
\end{theorem}	
\begin{proof}
Let $S$ be a set of generators for $G$.
Let $(f_h)_{h \in S}$ be generators for $G$ as a group of piecewise
affine maps. 
And consider the set of Wang tiles $(C, \phi, \psi)$ and maps $in, out_h$,
corresponding to them by Corollaries \ref{cor:fam} and \ref{cor:aff}.

Now we look at the set of Wang tiles $(C', (\phi_i), (\psi_i))$ over
$\mathbb{Z} \times G$ (with generators $1$ and $S$) defined by $C' = C$ and:
\[ \phi_1 = \phi \]	
\[ \psi_1 = \psi \]	
\[ \psi_{h} = in\]
\[ \phi_{h} = out_h\]

We will prove that (a) $X_{C'}$ mimics the behaviour of the piecewise
affine maps and (b) gives a strongly aperiodic SFT.

For an element $x \in X_{C'}$, $g \in G$, let $x_g : \mathbb{Z} \rightarrow C$ where $x_g(n) = x_{(n,g)}$.
Now let $z_g = cont(in(x_{g}))$.
%
%
Note that $x_g \in X_C$.

Note that by definition, for any $h$, 
\[\begin{array}{rcl}
	f_h(z_{g}) &=& f_h(cont(in(x_{g}))) \\
	&=& cont(out_h(x_{g})) \\
	&=& cont(\phi_h(x_g)) \\
	&=& cont(\psi_h(x_{gh^{-1}}))\\
	&=& cont(in(x_{gh^{-1}}))\\	
	&=& z_{gh^{-1}}
\end{array}
\]

This implies that for any $g= g_1\dots g_k$, we have
$z_{(g_1\dots g_k)^{-1}} = f_{g_1}(f_{g_2} \dots f_{g_k}(z_\lambda)\dots)$.
That is, for all $g \in G$, $z_{g^{-1}} = f_g(z_\lambda)$.

Now, by the second part of theorem \ref{thm:kari}, for any collection
$(z_g)_{g \in G}$ that satisfy
$z_{g^{-1}} = f_g(z_\lambda)$, there exists a corresponding configuration in $X_{C'}$.
This proves that the SFT is nonempty, by starting e.g. with 
$z_{g^{-1}} = f_g(0)$.

Now let $x \in X_{C'}$ and  $(n,h) \in \mathbb{Z} \times G$ be in the stabilizer of $x$, that
is for all $(m,g) \in \mathbb{Z} \times G$, we have
$((n,h)x)_{(m,g)} = x_{(m,g)}$.
This implies that $z_{gh^{-1}} = z_g$ for all $g$.
Hence, $f_h f_g(z_\lambda) = f_g(z_\lambda)$ for all $g$.
By PA-recognizability, this implies $h = \lambda$.

It remains to prove that $n = 0$.
If $n \not= 0$, this means that for each $g$, the word $x_g$ is
periodic of period $n$.
There are finitely many periodic words of length $n$, which means that
$z_g$ will take only finitely many values: $z_g \in Z$ for some finite
set $Z$, which is closed under all maps $f_h$.

Then each element of $G$ acts as a permutation on $Z$.
Furthermore, by PA-recognizability, any element of $G$ that acts like
the identity on $Z$ must be equal to the identity.
This implies that any element of $G$ is identified by the permutation
of $Z$ it induces, hence that $G$ is finite.
\end{proof}
\subsection{Applications}

\begin{proposition}
	$\mathbb{Z}$ is PA-recognizable. 
	Hence $\mathbb{Z} 	\times\mathbb{Z}$ admits a strongly aperiodic
	subshift of finite type
\end{proposition}	
\begin{proof}
The function $f$ seen previously provides a proof.
\[
	f(x) = \left\{\begin{array}{ccl}
	  (2x-1)/3 & \text{if }& 1/2 \leq x \leq 1\\
	  (4x+1)/3 & \text{if }& 0 \leq x \leq 1/2\\
\end{array}\right.\]

To understand better what $f$ does, we will look at $f' = hfh^{-1}$ where $h(x) = x + 1$.
Then it is easy to see that 
\[
	f'(x) = \left\{\begin{array}{ccl}
	  2x/3 & \text{if }& 3/2 \leq x \leq 2\\
	  4x/3 & \text{if }& 1 \leq x \leq 3/2\\
\end{array}\right.\]
from which it is easy to see that the orbit of $f$ is infinite, (hence
the group generated by $f$ is isomorphic to $\mathbb{Z}$), and that if
$f^n(t) =t$ for some $n$ and some $t$, then $n = 0$ (hence property (B)).
Therefore $G$ is $PA$-recognizable.
\end{proof}	
\begin{proposition}
Thompson group $T$ is PA-recognizable.
	Hence $\mathbb{Z} 	\times T$ admits a strongly aperiodic 	subshift of finite type
\end{proposition}	
\begin{proof}
$T$ is the quintessential PA-recognizable group: It is formally the 
subgroup of all piecewise affine maps of $[0,1]_{/0\sim 1}$ where each
affine map has dyadic coordinates and positive slope.

$T$ is indeed finitely generated, more precisely it is generated by
the three following functions \cite{CannonFloydParry}:

\[
	a(x) = \left\{\begin{array}{cl}
	  x/2 &  0 \leq x \leq 1/2\\
	  x-1/4 &  1/2 \leq x \leq 3/4\\
	  2x-1 &  3/4 \leq x \leq 1\\
\end{array}\right.
	b(x) = \left\{\begin{array}{cl}
	  x &  0 \leq x \leq 1/2\\
	  x/2 + 1/4 &  1/2 \leq x \leq 3/4\\
	  x-1/8 &  3/4 \leq x \leq 7/8\\
	  2x - 1 & 7/8 \leq x \leq 1\\
\end{array}\right.\]\[
c(x) = \left\{\begin{array}{cl}
	  x/2+3/4 &  0 \leq x \leq 1/2\\
	  2x-1 &  1/2 \leq x \leq 3/4\\
	  x - 1/4 & 3/4 \leq x \leq 1\\
\end{array}\right.
\]

From the definition of $T$, it is easy to see that the orbit of
any $z \in [0,1]$ is dense, hence property (B) is true, and $T$ is
PA-recognizable.

\end{proof}	

Note that it is not clear if Thompson group $F$ is PA-recognizable:
$F$ is the subgroup of $T$ generated by $a$ and $b$, and fixes $0$: 
As a consequence, this particular representation does not satisfy
property (B). Whether another representation of
Thompson group $F$ exists with this property is open.

\begin{proposition}
	$PSL_2(\mathbb{Z})$ is PA-recognizable.
	Hence $\mathbb{Z} \times PSL_2(\mathbb{Z})$ admits a
	strongly aperiodic subshift of finite type.
\end{proposition}	
\begin{proof}
$PSL_2(\mathbb{Z})$ is the subgroup of $T$ generated by $a$ and $c$,
see for example \cite{Fossas}.
To see that this representation satisfies property $(B)$, remark that
this action is conjugated by the Minkowski question mark symbol (ibid.)
to the action of $PSL_2(\mathbb{Z})$ over the projective line $\mathbb{R} \cup \{\infty\}$.
From this point of view, it is clear that any element of
$PSL_2(\mathbb{Z})$ different from the identity fixes at most two points. 
Hence if $gf(t) = f(t)$ for all $t$, then $g = \lambda$.

For this particular group, it is actually easy to work out all details
and produce a concrete aperiodic set of Wang tiles, represented in Fig~\ref{fig:psl2}. 
It is obtained by taking $d=a$ and $e=ac$ as generators (rather than $a$ and $c$)
and looking at them as acting on $[0,2]_{0\sim 2}$ (rather than
$[0,1]_{0\sim 1}$) by the formulas:

\[
	d(x) = \left\{\begin{array}{cl}
	  x/2 &  0 \leq x \leq 1\\
	  x-1/2 &  1 \leq x \leq 3/2\\
	  2x-2 &  3/2 \leq x \leq 2\\
\end{array}\right.
	e(x) = \left\{\begin{array}{cl}
	  x+1 &  0 \leq x \leq 1\\
	  x-1 &  1 \leq x \leq 2\\
\end{array}\right.\]

(Of course, such details may also be provided for Thompson
group $T$. However the presence of the generator $b$ produced an set
of tiles too large to be depicted here.)
\end{proof}

\newcommand\wangsl[5]{
        \draw[line width=0.2pt] (0,0) -- +(3,2) -- +(0,4) -- +(0,0);
        \draw[line width=0.2pt] (0,0) -- +(3,2) -- +(6,0) -- +(0,0);
\begin{scope}[xshift=6cm, yshift=4cm]
        \draw[line width=0.2pt] (0,0) -- +(-3,-2) -- +(0,-4) -- +(0,0);
        \draw[line width=0.2pt] (0,0) -- +(-3,-2) -- +(-6,0) -- +(0,0);
        \draw[line width=0.2pt] (0,0) -- +(-3,-2) -- +(-3,0) -- +(0,0);
\end{scope}%
                \draw (1,2) node {#1};
                \draw (5,2) node {#2};
                \draw (2,3.5) node {#3};
                \draw (4,3.5) node {#4};
                \draw (3,.5) node {#5};
}

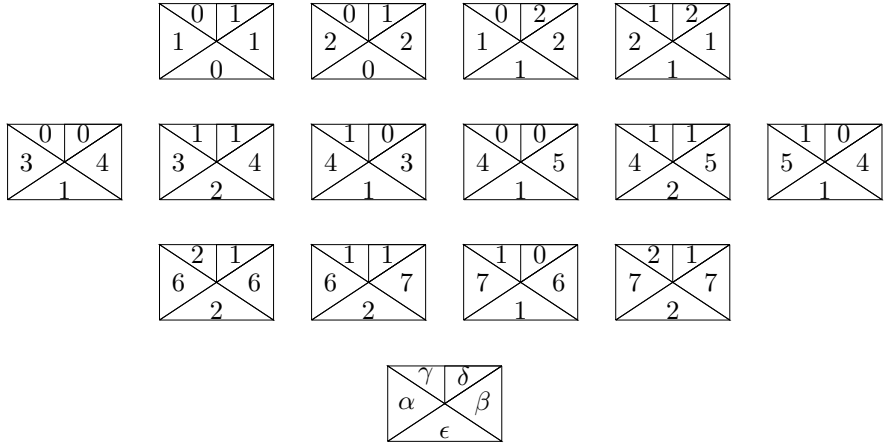
\begin{figure}[htbp]
	\input x3
\begin{center}	
	\begin{tikzpicture}[scale=0.25]
		\wangsl{$\alpha$}{$\beta$}{$\gamma$}{$\delta$}{$\epsilon$}
	\end{tikzpicture}
\end{center}	
\caption[A strongly aperiodic set of $14$ Wang tiles over $\mathbb{Z} \times PSL_2(\mathbb{Z})$]{A strongly aperiodic set of $14$ Wang tiles over $\mathbb{Z}  \times PSL_2(\mathbb{Z})$, where $PSL_2(\mathbb{Z})$ is generated by
  $d = \begin{pmatrix}  0 &-1\\  1 & 1 \end{pmatrix}$
 and $e = \begin{pmatrix} 0 & -1 \\ 1 & 0 \\ \end{pmatrix}$.
  The rules are as follows:
  Let $x$ be the tile in position $(n,g)$.
  Then the tile $y$ in position $(n+1,g)$ must satisfy $y_\alpha =
  x_\beta$, the tile $y$ in position $(n,gd)$ must satisfy $y_\epsilon
  =  x_\gamma$, the tile $y$ in position $(n,ge)$ must satisfy $y_\epsilon =  x_\delta$.
  }
\label{fig:psl2}
\end{figure}

\clearpage
\subsection{Generalizations}
The construction of Kari works for more than piecewise affine
homeomorphisms of $[0,1]$. It works for any partial piecewise affine
map from $[0,1]^d$ to its image.

\begin{theorem}[\cite{Kari5}]
	\label{thm:kari}
	Let $A \in M_{m\times n}(\mathbb{Q})$ be a (possibly non square) matrix with rational
	coefficients, $b \in \mathbb{Q}^m$ a rational vector and $f(x) = Ax + b$

	Then there exists a set of Wang tiles $(C,\phi,\psi)$ over $\mathbb{Z}$
	(generated by $1$) and two maps $out, in$ from $C$ to
	$\{0,1\}^m$ and $\{0,1\}^n$ so that the two following properties hold
	
	\begin{itemize}
		\item For any configuration $x$ of $X_C$, $f(cont_n(in(x)) = cont_m (out(x)$)
		  
		\item For any $y \in [0,1]^n$ so that $f(y) \in [0,1]^m$,
		  there exists a configuration $x$ of $C_G$ so that $in(x) = disc_n(y)$ and $out(x) = disc_m(f(y))$
    \end{itemize}		
	where $disc_i$ and $cont_i$ are the natural $i$-dimensional
	analogues of $disc$ and $cont$
\end{theorem}

Now we will be able to prove a theorem similar to the previous one for
a larger class of maps (hence a larger class of groups).
There are three directions in which we can go:
\begin{itemize}
	\item Go to higher dimensions
	\item Look at piecewise affine maps defined on compact subsets
	  of $\mathbb{R}^d$ different from $[0,1]^d$.
	\item Consider other identifications than $0\sim 1$
\end{itemize}	
In the following we will not use the full possible generalisation, and
will not identify any points in our sets. 
This will be sufficient for
the applications and already relatively painful to define.
However, this means that the next definition will not encompass
PA-recognizable groups.

\begin{defn}
	Let ${\cal F} = \{ f_i: B_i \mapsto B'_i, i =1 \dots k\}$ be 
	a finite set of piecewise affine rational homeomorphisms, where
	each $B_i$ and $B'_i$ is a finite union of bounded rational polytopes
	of $\mathbb{R}^n$.
	
	Let $S_{\cal F}$ be the closure of the set $f_i$ and
	$f_i^{-1}$ under composition.
	Each element of $S_{\cal F}$ is a piecewise affine homeomorphism,
	whose domain is the union of finitely many bounded rational
	polytopes, and may
	possibly be empty.

	Let $T_{\cal F}$ be the common domain of all functions in $S_{\cal F}$.

	Then the \emph{group} $G_{\cal F}$ generated by $\cal F$ is
	the group $\{f_{|T_{\cal F}}, f \in S_{\cal F}\}$.
\end{defn}

\begin{defn}
	A f.g. group $G$ is PA'-recognizable iff there exists a finite set
	$\cal F$ of piecewise affine rational homeomorphisms so that
	\begin{itemize}
	  \item	(A)	$G$ is isomorphic to $G_{\cal F}$.
	  \item (B) For any $t \in T_{\cal F}$, if $gf(t) = f(t)$ for all
		$f$, then $g = \lambda$
	\end{itemize}		
\end{defn}
Note that $T_{\cal F}$ might not be computable in general. In
particular, it is not clear that any PA'-recognizable has decidable
word problem.

\begin{theorem}
	If $G$ is PA'-recognizable, then the complement of the word
	problem on $G$ is recognizable. In particular, if $G$ is
	recursively presented, the word problem on $G$ is decidable
\end{theorem}
\begin{proof}
	We assume that $G \not= \{\lambda\}$, hence $T_{\cal F} \not=\emptyset$.
	
	Let $g$ be an element of $G$, given by composition of some
	piecewise affine maps.
    Let \[D = \left\{ t \middle| \forall f \in S_{\cal F}, f(t) \text{ is defined
		and } g(f(t)) = f(t)\right\}\]
	
	Note that $D \subseteq T_{\cal F}$.
	Furthermore, $g \not= \lambda$ iff $D = \emptyset$ by property $(B)$.
	
    This gives a semi algorithm to decide if $g \not=\lambda$.
\end{proof}	

\begin{theorem}
	If $G$ is PA'-recognizable, $\mathbb{Z} \times G$ admits a strongly
	aperiodic subshift of finite type.
\end{theorem}
\begin{proof}
	Same proof as before. 
\end{proof}	
Here a few applications:

\begin{proposition}
	$\mathbb{Z}$ is PA'-recognizable.
Hence $\mathbb{Z} \times \mathbb{Z}$ admits a strongly aperiodic
subshift of finite type.
\end{proposition}		

\begin{proof}
	Let $A = \left\{ (x,y) \in [-1,1]^2 , |x|+|y| \geq 1\right\}$.
	$A$ is the union of four bounded polytopes.
	
	Let \[
	\begin{array}{cccc}
	f: &A &\rightarrow& f(A)\\
&      \begin{pmatrix}
		  x\\y\\
		  \end{pmatrix}
		  &\mapsto &
		  \begin{pmatrix}
			  3/5 & 4/5 \\
			  -4/5 & 3/5\\
		  \end{pmatrix}			  
      \begin{pmatrix}
		  x\\y\\
		  \end{pmatrix}
\end{array}		  
	  \]
	  And let ${\cal F} = \{f\}$.
	$f$ is clearly an homeomorphism.
    Note that $f$ is a rotation of angle $\arccos 3/5$.

	Now it is easy to see that $T_{\{f\}} = S_1 = \{ (x,y) | x^2 + y^2 =
	  1\}$, and that $G_{\{f\}}$ is isomorphic to $\mathbb{Z}$.
	Furthermore, it is also clear that the orbit of any point of
	$T_{\{f\}}$ is dense in $T_{\{f\}}$, which implies property $(B)$.
	Hence  $\mathbb{Z}$	is PA'-recognizable.
\end{proof}	

\begin{proposition}
Any finitely generated subgroup $G$ of rational matrices of a compact
matrix group is PA'-recognizable. Hence $\mathbb{Z}  \times G$ admits a strongly aperiodic
subshift of finite type.

\end{proposition}	
\begin{proof}
	We assume familiary with representation theory of linear compact
	groups, see e.g.  \cite[Chap 3.4]{OniVin}.
	Let $G$ be such a group, and let $M_1 \dots M_n$ be the matrices
	of size $k \times k$ that generate $G$.
	Using elementary linear algebra we may suppose there exists a
	rational vector $y \in \mathbb{R}^k$ so that $gy =y \rightarrow g =
	\lambda$, and $Gy$ spans $\mathbb{R}^k$.

    Now, as $G$ is a subgroup of a compact group, we can define a
	scalar product so that all matrices of $G$ are unitary.
    Let $\mathbb{R}^k = V^1 \oplus V^2 \dots \oplus V^p$ be a
	decomposition of $\mathbb{R}^k$ into orthogonal (for this scalar
	product) irreducible $G$-invariant vector spaces, that is $GV^i = V^i$ and no
	proper nonzero subspace of
	$V^i$ is $G$-invariant. This is possible as $G$ is a subgroup
	of a compact group hence completely reducible.
	Note that the vector spaces $V^i$ might not have rational bases.
	
	Let $P^i$ be the orthogonal projection onto $V^i$.
    For a vector $x$, let $x^i = P^i x$, so that $x = \sum_i x^i$.
	For a matrix $g \in G$, let $g^i: V^i \rightarrow V^i$ be the restriction
	of $g$ to $V^i$, so that
	$gx = \sum_i g^i x^i$.
	
    Recall there is $y$ so that $gy = y$ for $g \in G$ iff $g = \lambda$.
	As $Gy$ spans $\mathbb{R}^k$, $y^i \not= 0$ for all $i$.
	
    Let $i\in \{1\dots p\}$.
	Let $W^i$ be the topological closure of $Gy^i$.
	As $y^i$ is nonzero and $V^i$ is $G$-invariant, $W^i \subseteq
	V^i$ and is faraway from zero. That is, there exists constants
	$r^i, R^i > 0$ so that for all $y \in W^i$, $|y|_1 > r^i$ and
	$|y|_1 < R^i$.
	
	Now let  \[T = \{ y | \forall i, |P^i y|_1 > r^i \text{ and } |P^i y|_1 <
	  R^i\}\] and
	\[ T_0 = \{ y | \forall i, |P^i y|_1 > r^i/2  \text{ and } |P^i
		  y|_1  < 	  2R^i\}\]
	
    Note that $T$ is a polytope with real coordinates. Let $T'$ be an
	approximation of $T$ as a polytope with rational coordinates, so that
	$T \subseteq T' \subseteq T_0$.
	
	Now define the maps $f_i$ as restrictions of $M_i$ from $T'$ to $M_i T'$.
	Let ${\cal F}$ be the corresponding set of maps.
	
    We cannot describe $T_{\cal F}$ exactly, but it is clear that it
	contains $y$, as $T$ contains the $G$-orbit of $y$.
	As a consequence, $G_{\cal F}$ is isomorphic to $G$.
	
    Now we prove property $(B)$. Start from $t \in T_{\cal F}$ and $g \in G_{\cal F}$ so that 
	${gf(t) = f(t)}$ for all $f$.
	
	Let $i \in \{1,\dots, p\}$ and let $t^i= P^i t$ so that $t = \sum_i t^i$.
	As $t \in T_{\cal F} \subseteq T' \subset T_0$, we have $t^i \not= 0$.
	As a consequence, the orbit of $G t^i$ on $W_i$ spans a
    nonzero	$G$-invariant subspace of $V_i$, which is $V_i$ by irreducibility.
	Now, as $gf(t) = f(t)$ for all $f$, we conclude that $g$ is the
	identity on the orbit of $Gt^i$, hence $g$ is the identity on $V^i$.
	As this is true for all $i$, $g$ is the identity matrix.
\end{proof}	
\begin{cor}
	The free group $\mathbb{F}_2$ is PA'-recognizable. Every finite group is
	PA'-recognizable.
\end{cor}	

\begin{proposition}
	Thompson's group $V$ is PA'-recognizable.	
\end{proposition}	
\begin{proof}
	$V$ is usually given \cite{CannonFloydParry} as the generalization of $T$ to discontinuous maps.
	However, our maps in the definition need to be continuous, so we will see
	$V$ as acting on the ``middle thirds'' Cantor set (As a side note, $V$ is therefore
	isomorphic to the group of all revertible generalized one-sided shifts
	\cite{Moore2}).
	
	Let \[ C_3 = \left\{ \sum_{i\geq 1} \frac{\alpha_i}{3^i}, \alpha \in
	  \{0,2\}^{\mathbb{N}^+}\right\}\]
	
	Let $a,b,c,\pi_0$ defined on $C_3$ by:
\[
	a(x) = \left\{\begin{array}{cl}
	  x/3 &  0 \leq x \leq 1/3\\ 
	  x - 4/9 &  2/3 \leq x \leq 7/9\\ 
	  3x-2 &  8/9 \leq x \leq 1\\  
\end{array}\right.
	b(x) = \left\{\begin{array}{cl}
	  x &  0 \leq x \leq 1/3\\ 
	  x/3 + 4/9 &  2/3 \leq x \leq 7/9\\ 
	  x-4/27 &  8/9 \leq x \leq 25/27\\ 
	  3x - 2 & 26/27 \leq x \leq 1\\ 
\end{array}\right.\]\[
c(x) = \left\{\begin{array}{cl}
	  x/3+8/9 &  0 \leq x \leq 1/3\\ 
	  3x-2 &  2/3 \leq x \leq 7/9\\ 
	  x - 2/9 & 8/9 \leq x \leq 1\\ 
\end{array}\right.
\pi_0(x) = \left\{\begin{array}{cl} 
  x/3 + 2/3 & 0 \leq x \leq 1/3 \\ 
  3x - 2 & 2/3 \leq x \leq 7/9 \\ 
  x & 8/9 \leq  x \leq  1\\ 
\end{array}\right.
\]

Now our definition does not permit to define $a,b,c,\pi_0$ on $C_3$,
as the domain and range of each map should be a finite union of
intervals with rational coordinates.
So we will define them by the above formulas, but for $x \in [0,1]$
rather than $x \in C_3$.
Note that they are already homeomorphisms onto their image.

Let ${\cal F} = \{ a,b,c,\pi_0\}$.
We claim that $T_{\cal F} = C_3$, which will prove that $G_{\cal F}$
is indeed isomorphic to $V$. As before, any orbit is dense, from which 
property $(B)$ ensues and $V$ will be PA'-recognizable.

It remains to prove that $T_{\cal F} = C_3$.
Note that clearly $C_3 \subseteq T_{\cal F}$.

First note that 
\begin{itemize}
	\item $Dom(a) = [0,1/3] \cup [2/3,7/9]\cup [8/9,1]$
    \item $Range(a) = [0,1/9] \cup [2/9,1/3] \cup [2/3,1]$
\end{itemize}
Which implies that $T_{\cal F} \subseteq [0,1/9] \cup [2/9,1/3] \cup [2/3,7/9] \cup [8/9,1]$

Now let $x \in T_{\cal F}$.
\begin{itemize}
	\item If $0 \leq x \leq 1/9$, then $3x \in T_{\cal F}$ (apply $a^{-1})$
	\item if $2/9 \leq x \leq 1/3$, then $3x \in T_{\cal F}$  (apply $a^{-1}$, then $c^{-1}$ then $a$)
	\item if $2/3 \leq x \leq 7/9$, then $3x - 2 \in T_{\cal F}$ (apply $c$)
	\item If $8/9 \leq x \leq 1$, then $3x - 2 \in T_{\cal F}$ (apply $a$)
\end{itemize}	
This proves inductively that $x \in C_3$.
\end{proof}	
\section*{Open Problems}

This is only one way of generalizing Kari's construction.
There are many other ways to generalize it, one of which providing a
(weakly) aperiodic SFT on the Baumslag Solitar group, see \cite{AuKa}.

Here is an interesting open question:
 The construction uses representations of reals as words in
 $\{0,1\}^\mathbb{Z}$, can we use a representation in
 $\{0,1\}^H$, for some other group $H$ ?
 This would possibly allow to prove that $H \times G$ has a strongly
 aperiodic SFT for $G$ PA-recognizable.


\end{document}

%% file: x1.tex
\begin{center}
\begin{tikzpicture}[scale=0.25]
\begin{scope}[xshift=0cm, yshift=0cm]
\wang{0}{1}{1}{0}
\end{scope}
\begin{scope}[xshift=6cm, yshift=0cm]
\wang{1}{0}{0}{0}
\end{scope}
\begin{scope}[xshift=12cm, yshift=0cm]
\wang{1}{2}{1}{0}
\end{scope}
\begin{scope}[xshift=18cm, yshift=0cm]
\wang{2}{1}{0}{0}
\end{scope}
\begin{scope}[xshift=24cm, yshift=0cm]
\wang{2}{3}{1}{0}
\end{scope}
\begin{scope}[xshift=30cm, yshift=0cm]
\wang{2}{0}{1}{1}
\end{scope}
\begin{scope}[xshift=36cm, yshift=0cm]
\wang{3}{2}{0}{0}
\end{scope}
\begin{scope}[xshift=42cm, yshift=0cm]
\wang{3}{1}{1}{1}
\end{scope}
\end{tikzpicture}
\end{center}

%% file: x2.tex
\begin{center}
\begin{tikzpicture}[scale=0.25]
\begin{scope}[xshift=0cm, yshift=0cm]
\wang{0}{1}{0}{0}
\end{scope}
\begin{scope}[xshift=6cm, yshift=0cm]
\wang{0}{3}{1}{0}
\end{scope}
\begin{scope}[xshift=12cm, yshift=0cm]
\wang{0}{2}{1}{1}
\end{scope}
\begin{scope}[xshift=18cm, yshift=0cm]
\wang{1}{2}{0}{0}
\end{scope}
\begin{scope}[xshift=24cm, yshift=0cm]
\wang{1}{4}{1}{1}
\end{scope}
\begin{scope}[xshift=30cm, yshift=0cm]
\wang{2}{4}{0}{0}
\end{scope}
\begin{scope}[xshift=36cm, yshift=0cm]
\wang{2}{0}{0}{1}
\end{scope}
\begin{scope}[xshift=42cm, yshift=0cm]
\wang{2}{5}{1}{1}
\end{scope}
\begin{scope}[xshift=0cm, yshift=6cm]
\wang{3}{4}{0}{1}
\end{scope}
\begin{scope}[xshift=6cm, yshift=6cm]
\wang{4}{5}{0}{0}
\end{scope}
\begin{scope}[xshift=12cm, yshift=6cm]
\wang{4}{1}{0}{1}
\end{scope}
\begin{scope}[xshift=18cm, yshift=6cm]
\wang{4}{3}{1}{1}
\end{scope}
\begin{scope}[xshift=24cm, yshift=6cm]
\wang{5}{3}{0}{0}
\end{scope}
\begin{scope}[xshift=30cm, yshift=6cm]
\wang{5}{2}{0}{1}
\end{scope}
\end{tikzpicture}
\end{center}

%% file: x3.tex
\begin{center}
\begin{tikzpicture}[scale=0.25]
\begin{scope}[xshift=0cm, yshift=0cm]
\wangsl{1}{1}{0}{1}{0}
\end{scope}
\begin{scope}[xshift=8cm, yshift=0cm]
\wangsl{2}{2}{0}{1}{0}
\end{scope}
\begin{scope}[xshift=16cm, yshift=0cm]
\wangsl{1}{2}{0}{2}{1}
\end{scope}
\begin{scope}[xshift=24cm, yshift=0cm]
\wangsl{2}{1}{1}{2}{1}
\end{scope}
\end{tikzpicture}
\end{center}
\begin{center}
\begin{tikzpicture}[scale=0.25]
\begin{scope}[xshift=0cm, yshift=0cm]
\wangsl{3}{4}{0}{0}{1}
\end{scope}
\begin{scope}[xshift=8cm, yshift=0cm]
\wangsl{3}{4}{1}{1}{2}
\end{scope}
\begin{scope}[xshift=16cm, yshift=0cm]
\wangsl{4}{3}{1}{0}{1}
\end{scope}
\begin{scope}[xshift=24cm, yshift=0cm]
\wangsl{4}{5}{0}{0}{1}
\end{scope}
\begin{scope}[xshift=32cm, yshift=0cm]
\wangsl{4}{5}{1}{1}{2}
\end{scope}
\begin{scope}[xshift=40cm, yshift=0cm]
\wangsl{5}{4}{1}{0}{1}
\end{scope}
\end{tikzpicture}
\end{center}
\begin{center}
\begin{tikzpicture}[scale=0.25]
\begin{scope}[xshift=0cm, yshift=0cm]
\wangsl{6}{6}{2}{1}{2}
\end{scope}
\begin{scope}[xshift=8cm, yshift=0cm]
\wangsl{6}{7}{1}{1}{2}
\end{scope}
\begin{scope}[xshift=16cm, yshift=0cm]
\wangsl{7}{6}{1}{0}{1}
\end{scope}
\begin{scope}[xshift=24cm, yshift=0cm]
\wangsl{7}{7}{2}{1}{2}
\end{scope}
\end{tikzpicture}
\end{center}

%% file: groups.bbl
\begin{thebibliography}{Wan61}

\bibitem[ABS]{AubBar}
Nathalie Aubrun, Sebastián Barbieri, and Mathieu Sablik.
\newblock {A notion of effectiveness for subshifts on finitely generated
  groups}.
\newblock arXiv:1412.2582.

\bibitem[AK13]{AuKa}
Nathalie Aubrun and Jarkko Kari.
\newblock {Tiling Problems on Baumslag-Solitar groups}.
\newblock In {\em Machines, Computations and Universality (MCU)}, number 128 in
  Electronic Proceedings in Theoretical Computer Science, pages 35--46, 2013.

\bibitem[AS09]{AubrunS09}
Nathalie Aubrun and Mathieu Sablik.
\newblock An order on sets of tilings corresponding to an order on languages.
\newblock In {\em 26th International Symposium on Theoretical Aspects of
  Computer Science, STACS 2009, February 26-28, 2009, Freiburg, Germany,
  Proceedings}, pages 99--110, 2009.

\bibitem[CFP96]{CannonFloydParry}
James~W. Cannon, William~J. Floyd, and Walter~R. Parry.
\newblock {Introductory Notes on Richard Thompson's Groups}.
\newblock {\em L'Enseignement Mathématique}, 42:215--216, 1996.

\bibitem[Coh14]{Cohen2014}
David~Bruce Cohen.
\newblock {The large scale geometry of strongly aperiodic subshifts of finite
  type}.
\newblock arXiv:1412.4572, 2014.

\bibitem[CP]{Carroll}
David Carroll and Andrew Penland.
\newblock {Periodic Points on Shifts of Finite Type and Commensurability
  Invariants of Groups}.
\newblock arXiv:1502.03195.

\bibitem[CSC10]{CicCoo}
Tullio Ceccherini-Silberstein and Michel Coornaert.
\newblock {\em {Cellular Automata on Groups}}.
\newblock Springer Monographs in Mathematics. Springer, 2010.

\bibitem[Fos11]{Fossas}
Ariadna Fossas.
\newblock {$PSL(2,\ZZ)$ as a Non-distorted Subgroup of Thompson's Group $T$}.
\newblock {\em Indiana University Mathematics Journal}, 60(6):1905--1925, 2011.

\bibitem[FR59]{FriedbergRogers}
Richard~M. Friedberg and Hartley Rogers.
\newblock {Reducibility and Completeness for Sets of Integers}.
\newblock {\em Zeitschrift für mathematische Logik und Grundlagen der
  Mathematik}, 5:117--125, 1959.

\bibitem[GJS09]{Gao}
Su~Gao, Steve Jackson, and Brandon Seward.
\newblock {A coloring property for countable groups}.
\newblock {\em Mathematical Proceedings of the Cambridge Philosophical
  Society}, 147:579--592, 2009.

\bibitem[HS88]{HigmanScott}
Graham Higman and Elizabeth Scott.
\newblock {\em {Existentially Closed Groups}}.
\newblock Oxford University Press, 1988.

\bibitem[Kar96]{Kari14}
Jarkko Kari.
\newblock {A small aperiodic set of Wang tiles}.
\newblock {\em Discrete Mathematics}, 160:259--264, 1996.

\bibitem[Kar07]{Kari5}
Jarkko Kari.
\newblock {The Tiling Problem Revisited}.
\newblock In {\em Machines, Computations, and Universality (MCU)}, number 4664
  in Lecture Notes in Computer Science, pages 72--79, 2007.

\bibitem[Lin04]{LindMulti}
Douglas~A. Lind.
\newblock {Multi-Dimensional Symbolic Dynamics}.
\newblock In Susan~G. Williams, editor, {\em Symbolic Dynamics and its
  Applications}, number~60 in Proceedings of Symposia in Applied Mathematics,
  pages 61--79. American Mathematical Society, 2004.

\bibitem[LM95]{LindMarcus}
Douglas~A. Lind and Brian Marcus.
\newblock {\em An Introduction to Symbolic Dynamics and Coding}.
\newblock Cambridge University Press, New York, NY, USA, 1995.

\bibitem[Moo91]{Moore2}
Cristopher Moore.
\newblock {Generalized one-sided shifts and maps of the interval}.
\newblock {\em Nonlinearity}, 4(3):727--745, 1991.

\bibitem[Moz97]{Mozes:1997}
Shahar Mozes.
\newblock {Aperiodic tilings}.
\newblock {\em Inventiones mathematicae}, 128:603--611, 1997.

\bibitem[Odi99]{Odifreddi2}
P.G. Odifreddi.
\newblock {\em {Classical Recursion Theory Volume II}}, volume 143 of {\em
  Studies in Logic and The Foundations of Mathematics}.
\newblock North Holland, 1999.

\bibitem[Ol'83]{Olshanskii}
A.~Yu Ol'shanskii.
\newblock {Groups of bounded period with subgroups of prime order }.
\newblock {\em Algebra and Logic}, 21(5):369--418, 1983.

\bibitem[Ol'91]{Olshanskiibook}
A.~Yu Ol'shanskii.
\newblock {\em {Geometry of defining relations in groups}}.
\newblock Kluwer Academy Publisher, 1991.

\bibitem[Osi10]{Osin}
Denis Osin.
\newblock {Small cancellations over relatively hyperbolic groups and embedding
  theorems}.
\newblock {\em Annals of Mathematics}, 172:1--39, 2010.

\bibitem[OV90]{OniVin}
A.L. Onishchik and E.B. Vinberg.
\newblock {\em {Lie Groups and Algebraic Groups}}.
\newblock Springer Series in Soviet Mathematics. Springer, 1990.

\bibitem[Rob71]{Robinson}
Raphael~M. Robinson.
\newblock {Undecidability and Nonperiodicity for Tilings of the Plane}.
\newblock {\em Inventiones Mathematicae}, 12(3):177--209, 1971.

\bibitem[Wan61]{wangpatternrecoII}
Hao Wang.
\newblock {Proving theorems by Pattern Recognition II}.
\newblock {\em Bell Systems technical journal}, 40:1--41, 1961.

\end{thebibliography}
